\documentclass[11pt,a4paper,reqno]{amsart}
\usepackage{amsmath,amssymb,amsthm}
\usepackage[utf8]{inputenc}
\usepackage{enumitem}
\usepackage{graphicx}
\usepackage{tikz}

\usepackage[alphabetic]{amsrefs}

\usepackage{todonotes}
\usepackage[capitalise,noabbrev]{cleveref}
\usepackage{mathtools}

\DefineSimpleKey{bib}{primaryclass}{}
\DefineSimpleKey{bib}{archiveprefix}{}

\BibSpec{arXiv}{%
  +{}{\PrintAuthors}{author}
  +{,}{ \textit}{title}
  +{}{ \parenthesize}{date}
  +{,}{ arXiv }{eprint}
  +{,}{ primary class }{primaryclass}
}

\setenumerate{label = \rm (\roman*)}

\DeclareMathOperator{\End}{End}
\DeclareMathOperator{\Hom}{Hom}

\newcommand*{\ZZ}{\mathbb{Z}}

\newcommand*{\Z}{\mathbb{Z}}
\newcommand*{\vacuum}{\mathbf{1}}
\renewcommand{\L}[1]{\mathsf{L}_{#1}}

\renewcommand{\arraystretch}{1.5}
\newcommand{\dd}[1]{\frac{\mathrm{d}}{\mathrm{d}#1}}
\newcommand{\D}[1]{\mathcal{D}^{\left(#1\right)}}

\numberwithin{equation}{section}
\newtheorem{lemma}[equation]{Lemma}
\newtheorem{proposition}[equation]{Proposition}

\newtheorem{theorem}[equation]{Theorem}
\theoremstyle{definition}
\newtheorem{definition}[equation]{Definition}

\theoremstyle{remark}
\newtheorem{remark}[equation]{Remark}

\newcommand*{\ViiA}{a}
\newcommand*{\ViiB}{b}
\newcommand*{\Vir}{\mathsf{Vir}}
\newcommand*{\iv}{e} 
\newcommand*{\ccv}{\omega} 

\newcommand{\diagnode}[2]{\fill #1 circle (.1) [above] node{\tiny #2}; }

\newcommand*{\ViA}{d}
\newcommand*{\ViB}{e}

\title{Decompositions of Griess algebras beyond the OZ-setting}
\author{Jari Desmet}
\author{Louis Olyslager}
\makeatletter
\@namedef{subjclassname@2020}{%
  \textup{2020} Mathematics Subject Classification}
\makeatother

\begin{document}
\maketitle
\begin{abstract} To each VOA of strong CFT type, we associate a non-associative algebra which generalizes Griess algebras of One-Zero (OZ) VOAs. We compute the fusion law of Ising vectors in this algebra using techniques from Miyamoto, leveraging the module theory of the corresponding Virasoro VOA.  This new class of algebras encapsulates the Chayet--Garibaldi algebras, a family of algebras constructed from absolutely simple linear algebraic groups that includes the $3876$-dimensional algebra for $E_8$. 
\end{abstract}

\section{Introduction}
Chayet--Garibaldi algebras were introduced in \cite{CG21} to provide an explicit construction of the unique $3876$-dimensional algebra for algebraic \linebreak groups of type $E_8$, which was shown to exist in \cite{GG15}. In \cite{DPhD}, the first-named author observed a striking similarity between certain idempotents in these algebras and idempotents in the Conway--Griess--Norton algebra, the $196884$-dimensional algebra used to construct the Monster group \cites{G82,C85}.  Famously, the  Conway--Griess--Norton algebra is intimately connected to the theory of \emph{vertex operator algebras}, and Masahiko Miyamoto used it as a guiding example to define \emph{Griess algebras} in \cite{M96}.

Unexpectedly, the second-named author, joint with Tom De Medts, was able to construct the Chayet--Garibaldi algebras from vertex operator algebras in \cite{DMO}, revealing a concrete connection between Chayet--Garibaldi algebras and Griess algebras. 

In this article, we provide a natural framework for this construction by generalizing Griess algebras to VOAs of strong CFT type (\cref{Generalized Griess algebra}), and show that Chayet--Garibaldi algebras are examples coming from irreducible affine VOAs (\cref{prop:CGisgriess}). 

Using this new setting, we can easily describe the relation between idempotents in Chayet--Garibaldi algebras and \emph{Ising vectors}. In \cite{M96}, the author observed that Ising vectors $\iv \in V$, vectors generating a Virasoro VOA $\mathrm{Vir}(\iv)\subseteq V$ of central charge $\frac12$,  define automorphisms of an OZ VOA $V$, by analyzing $V$ as a module for $\mathrm{Vir}(\iv)$. Miyamoto's observation and Sakuma's subsequent work \cite{S07} sparked the study of \emph{Majorana algebras} (\cite{I09}) and \emph{axial algebras} (\cite{HRS15}). In the same vein as Miyamoto, we show that Ising vectors of a VOA of strong CFT type correspond to idempotents of its generalized Griess algebra. 

An idempotent $e$ defines a \emph{fusion law}, a binary map $\star: S\times S \to 2^S$ on its spectrum $S$, 
 specifying that the product of $\lambda$- and $\mu$-eigenvectors lies in the sum of the $\rho$-eigenspaces with $\rho \in \lambda\star\mu$. We compute the fusion law of these idempotents in the sense of axial algebras. Surprisingly, only one more eigenvalue shows up in the spectrum, compared to the spectrum of the Griess algebra. We believe that the resulting two fusion laws, given in \cref{table:fl}, are natural objects to study in the framework of axial algebras.
\begin{table}
	\[	
	\begin{array}{c||c|c|c|c|c}
	
 *          & 0               & 1            & \frac14       & \frac1{32}   &     \frac12                     \\\hline\hline
 0          & 0         &              & \frac14,\frac12      & \frac1{32},\frac12 & \frac1{32},\frac14,\frac12            \\\hline
 1          &                   & 1    & \frac14        & \frac1{32}   & \frac12     \\\hline
 \frac14   & \frac14,\frac12    & \frac14        & 0,1            & \frac1{32},\frac12 & 0,\frac1{32},\frac12             \\\hline
 \frac1{32} & \frac1{32},\frac12      & \frac1{32}    & \frac1{32},\frac12   & 0,\frac14,1  & 0,\frac14,\frac12            \\\hline
 \frac12          & \frac1{32},\frac14,\frac12 & \frac12  & 0,\frac1{32},\frac12 & 0,\frac14,\frac12  & 0,\frac14,\frac1{32},1                         
 
	\end{array} 
\hspace{6ex}
\begin{array}{c||c|c|c|c}
	
	*          & 0     & 1        & \frac14   & \frac12               \\\hline\hline
	0          & 0    &          & \frac14,\frac12 & \frac14,\frac12        \\\hline
	1          &        & 1   & \frac14   & \frac12          \\\hline
	\frac14    & \frac14,\frac12 & \frac14  & 0,1       & 0           \\\hline
	\frac12          & \frac14,\frac12 & \frac12   & 0         & 0,1              
	
\end{array} 
\]
\vspace{1em}	
	\caption{\small This table depicts the fusion law of idempotents in Chayet--Garibaldi algebras corresponding to Ising vectors. The left is the general case, while the right is for Ising vectors of $\sigma$-type. Removing the values $\frac{1}{2}$ from the table produces the Monster fusion law (left) and the Jordan fusion law (right).}\label{table:fl}
\end{table}

\subsection*{Outline of the paper}
In the preliminaries (Section 2) we recall the definitions and basic properties of vertex operator algebras and their modules. A particular emphasis is placed on the module theory of the Virasoro VOA. We  recall basic computational tools of VOAs (\cref{lem:formulas}) as well as a summary of the representation theory of Virasoro VOAs of central charge $\frac{1}{2}$ (\cref{prop:fuse}), which allows us to easily compute the fusion law in Section 4.
 
In the rather short Section 3, we define the generalized Griess algebra $G(V)$ of a vertex operator algebra $V$ of strong CFT type. Our definition is natural, based on a construction in \cite{B86}, and $G(V)$ is the usual Griess algebra when $V$ is OZ. We also show that the algebra $G(V)$ is a Chayet--Garibaldi algebra when $V = L_{\hat{\mathfrak g}}(1,0)$ is a simple affine vertex operator algebra with $\mathfrak g$ a simple Lie algebra (\cref{prop:CGisgriess}).

Section 4 is the bulk of the paper. We exploit the module theory of a Virasoro VOA of central charge $\frac{1}{2}$ to compute the spectrum (\cref{lemma:eigenspaces}) and the fusion law (\cref{lem:stdfusion}) of an Ising vector in the generalized Griess algebra. As a side result, we show that any Ising vector of $V$ corresponds to an idempotent of $G(V)$ (\cref{lemma:42}). We also compute the fusion law for an Ising vector of $\sigma$-type (\cref{prop:fuslawBn}), and show that both Ising vectors of $\sigma$-type and not of $\sigma$-type occur in Chayet--Garibaldi algebras.
 \subsection*{Assumptions} Throughout the paper, $\mathbb{F}$ is a field of characteristic different from $2$ and $7$. While $2$ being invertible is unavoidable, characteristic $7$ is excluded because the module theory of the Virasoro VOA of central charge $\frac{1}{2}$ is not fully understood, cf.\@ \cite{DR2}. In Section 4, we additionally assume $\operatorname{char} \mathbb{F} \neq 3,31$, though we remark on how to resolve these characteristics in \cref{rem:kar31}. Throughout, whenever we mention a vertex operator algebra, it is always assumed to be of strong CFT type.
 \subsection*{Acknowledgments}
The second author is supported by the FWO PhD mandate 1105425N. We are grateful to Tom De Medts for interesting discussions and significantly improving the exposition of this article.

\section{Preliminaries}
In this section, based on \cite{DR2}, we repeat the necessary definitions and basic properties regarding VOAs. This is not a comprehensive introduction to VOAs. We refer to \cite{LL04,K97} for excellent introductions to the topic.

\subsection{Vertex operator algebras} For now, let $\mathbb{F}$ be a field of characteristic not $2$. Later, we additionally assume that the characteristic is not $7$.

\begin{definition}
	The Virasoro Lie algebra over a field $\mathbb{F}$ is the Lie algebra $\Vir$ generated by $\{\L n |n\in\ZZ\}\cup \{\mathbf c\}$ where $\mathbf c$ is a central element central and \begin{equation}\label{eq:vir} [\L m ,\L n] = (m-n)\L{m+n}+\delta_{m+n,0}\frac{m^3-m}{12}\mathbf{c}.
		 \end{equation}
\end{definition}

\begin{definition}
A \emph{vertex operator algebra of strong CFT type} over $\mathbb{F}$ is a quadruple $(V,Y,\vacuum, \ccv)$ consisting of the following data. The space $V$ is an $\mathbb{N}$-graded $\mathbb{F}$-vector space $V = \oplus_{n\in\mathbb{N}}V_n$ with $V_0$ one-dimensional and the other graded components finite dimensional. This space contains two distinguished elements: $\vacuum \in V_0$ called the vacuum vector and $\ccv\in V_2$ called the Virasoro vector. Furthermore,  $Y\colon V\to \End(V)[[z,z^{-1}]]$ is a linear map, where $z$ is a formal variable. 
For $v\in V$ we use $v_n$ to denote the coefficient of $z^{-n-1}$ in $Y(v,z)$: 
\[Y(v,z) = \sum_{n\in\ZZ} v_nz^{-n-1}.\] These $v_n$ are often called \emph{modes}. 

These data should satisfy the following axioms for all $u,v\in V$:
	\begin{itemize}[leftmargin=4em]
		\item[(VA1)] $u_nv = 0$ for $n > N$ for some $N$ depending on $u$ and $v$,
		\item[(VA2)] $Y(\vacuum, z) = \mathrm{id}_V$,
		\item[(VA3)] $v_{-1} \vacuum = v$ and $v_{n} \vacuum = 0$  for all $n \geq 0$,
		\item[(VA4)] Borcherds' identity:
		\begin{multline*}\label{eq:VOA-borcherds}
			  z_0^{-1}\delta\left(\frac{z_1-z_2}{z_0}\right)  Y(u,z_1)Y(v,z_2) 
			- z_0^{-1}\delta\left(\frac{z_2-z_1}{-z_0}\right) Y(v,z_2)Y(u,z_1) \\
			= z_2^{-1}\delta\left(\frac{z_1-z_0}{z_2}\right)  Y(Y(u,z_0)v,z_2);
		\end{multline*}
		\item[(VOA1)]  The modes $\L n=\ccv_{n+1}$ of $\ccv$ form a Virasoro Lie algebra where the central element $\mathbf c$ should act as a scalar $c\in \mathbb{F}$ on $V$. This scalar $c$ is called the \emph{central charge} of $V$;
		\item[(VOA2)] $\L0v = mv$ for $m\in \ZZ$ and $v\in V_m$;
		\item[(VOA3)] The map $\L{-1}$ acts as derivations $Y(\L{-1}v,z) = \dd z Y(v,z)$ for $v\in V$.
		\item[(VOA4)]  $u_nv\in V_{r+s+n-1}$ for $u\in V_r,v\in V_s$.
		\item[(CFT1)] $\L1V_1 =0$.
	\end{itemize}
\end{definition}
We call a VOA \emph{OZ} (\emph{one-zero}) if $V_1$ is the zero space. Note that (VOA4) follows from the other axioms in characteristic zero, since then $V_m $ is equal to $ \{v \in V \mid \L0v = mv \}$. Unfortunately, this is no longer true in positive characteristic. 

VOAs of strong CFT type were introduced in \cite{DM04}. In characteristic zero (CFT1) is equivalent to $V$ admitting an invariant bilinear form in the sense of \cite{FHL} by the work of Li \cite{Li94}. In positive characteristic however, this is a weaker condition, cf.\@ \cite{LQ18}.

Some consequences of Borcherds' identity are compiled in the following lemma, which we will use throughout.
For $V$  a VOA, we define the operators $\D n$ by $\D nv\coloneqq v_{-1-n}\vacuum \ \text{for}\ n \in \ZZ$, for all $v\in V$. We will write $\mathcal D$ in place of $\D1$.
\begin{lemma}\label{lem:formulas} Let $V$ be a VOA of strong CFT type and $u,v\in V$, then the following properties hold :

\begin{enumerate}
	\item skew-symmetry: \[u_nv = \sum_{i\geq0} (-1)^{i+n+1} \D i(v_{n+i}u);\]
	\item iterate formula:  \[(u_mv)_n = \sum (-1)^i\binom{m}{i}(u_{m-i}v_{n+i}-(-1)^mv_{m+n-i}u_i);\]
	\item commutator formula: \[[u_m,v_n] = \sum_{i\geq0} \binom{m}{i} (u_iv)_{m+n-i}.\]
	\item $\mathcal D$-derivative property: \[[\mathcal D, v_n] = [\L{-1},v_n] = (\mathcal D v)_n = (\L{-1}v)_n=-nv_{n-1}\]
\end{enumerate}
\end{lemma}
The formulas are a consequence of Borcherds' axioms of vertex algebras, see \cite[Section 4]{B86} and \cite{DR2}. Specifically for VOAs of strong CFT type, we also have the following fact which we will use throughout.
\begin{lemma}\label{L1V10}
	Let $a,b\in V_n$ be homogeneous of weight $n>0$, and $k\in \Z$. We have
	\begin{enumerate}
		\item  $a_{2n-1}b = b_{2n-1}a $:
		\item	$a_{2n-2}b = -b_{2n-2}a$ ;
		\item If $\L1a=0$, then $[\L1, a_k] =  (2n - k -2)a_{k+1}$. 
	\end{enumerate}
\end{lemma}
\begin{proof}
	The first claim follows from skew-symmetry (\cref{lem:formulas}(i)) and the fact that $a_kb$ is of negative weight for $k\geq 2n$ by (VOA4), hence zero.
	
	The second claim follows from the fact that $\L{-1}\vacuum= \ccv_0 \vacuum = 0$ by (VA3). Indeed, if $a,b\in V_1$, then $a_1b\in V_0 = \mathbb{F}1$. By skew-symmetry, we have $b_{2n-2}a = (-1)^{2n-1}(a_{2n-2}b - \mathcal{D}(a_{2n-1}b))$. In this expression, the terms involving $\D i$ with $i>1$ disappear since $V$ is nonnegatively graded, and $\mathcal{D}(a_{2n-1}b)=0$ since $\mathcal{D}(a_{2n-1}b) = \L{-1}(a_{2n-1}b) = 0$.
	
	The third claim follows from the commutator formula (\cref{lem:formulas}(iii)):
	\begin{align*}
		[\L1,a_k] = [\ccv_2,a_k] &= (\ccv_0a)_k+ 2(\ccv_1a)_{k+1} + (\ccv_2a)_k\\
		&= (\L{-1}a)_{k+2}+2na_{k+1} + 0.
	\end{align*}
	Applying the $\mathcal{D}$-derivative property (\cref{lem:formulas}(iv)) to the right-hand side proves the claim.
 \end{proof}

We will consider the structure of a VOA $V$ as a module over a vertex subalgebra 
$W$. To do so, we define both vertex subalgebras and admissible modules.
\begin{definition} A VOA $(W,Y_{W},\vacuum_{W},\ccv_W)$ is a \emph{vertex subalgebra} of a VOA $(V,Y,\vacuum, \ccv)$, if $W$ is a subspace of $V$, for all $u,v\in W$ the series $Y_W(u,z)v$ and $Y(u,z)v$ coincide, and $\vacuum_{W}$ equals $\vacuum$.
\end{definition}

\begin{definition}
An \emph{admissible module} for a VOA $V$ is an $\mathbb{N}$-graded vector space $M = \bigoplus_{n\in\mathbb{N}}M_n$ over $\mathbb{F}$ with a linear map 
\[Y_M\colon V\to \End(M)[[z,z^{-1}]]:v\mapsto Y_M(v,z)=\sum_{n\in\ZZ}v_nz^{-1-n},\] satisfying the following axioms for $w\in M$ and $u,v\in V$:
\begin{itemize}[leftmargin=4em]
		\item[(M1)] $v_nw = 0$ for $n$ sufficiently large;
		\item[(M2)] $Y_M(\vacuum, z) = \mathrm{id}_M$;
		\item[(M3)] Borcherds' identity
		\begin{multline*}\label{eq:VOA-borcherds}
			  z_0^{-1}\delta\left(\frac{z_1-z_2}{z_0}\right)  Y_W(u,z_1)Y_W(v,z_2) 
			- z_0^{-1}\delta\left(\frac{z_2-z_1}{-z_0}\right) Y_W(v,z_2)Y_W(u,z_1) \\
			= z_2^{-1}\delta\left(\frac{z_1-z_0}{z_2}\right)  Y_W(Y(u,z_0)v,z_2);
		\end{multline*} 
		\item[(M4)] $v_mM_{n} \in M_{n+s-m-1}$ for $v\in V_s$.
	\end{itemize}
\end{definition}
We use the same notation $v_n$ for the operator in $\End(V)$ and the operator in $\End(W)$. It should be clear from the context how to interpret this notation. 

\subsection{The Virasoro VOA of central charge $\frac12$} Now we introduce the simplest and one of the key VOAs for this article. 
Recall the Lie algebra $\Vir$ with basis $\{\L n| n\in \Z\}\cup\{\mathbf c\}$ satisfying the Virasoro relations \eqref{eq:vir}. The Lie subalgebra spanned by $\L i$, $i\geq0$ is denoted by $\Vir^+$. From it we can construct a VOA and its modules as follows.

 Write $U(\Vir)$ for the universal enveloping algebra of $\Vir$ and let $c,h\in \mathbb{F}$ be scalars. Construct the $\Vir^+$-module
\[ V(c,h) \coloneqq U(\Vir)\otimes_{U(\Vir^+)}\mathbb F,\] where $\L0$ acts on $\mathbb F$ as multiplication by $h$, $\mathbf{c}$ by multiplication by $c$, and $\L n$ (for $n > 0$) as multiplication by $0$.  Then $V(c,h)$ is also a $\Vir$-module, by letting $\Vir$ act on the left side of the tensor product. The following theorem endows these spaces with a module structure for a VOA.

\begin{theorem}[\cite{DR16}]\label{thm:virvoa} Suppose $c\neq 0,h\in \mathbb F$. Then:
	\begin{enumerate}
		\item $V(c,h)$ can be endowed with a grading, and it has a unique maximal graded $\Vir$-submodule. The quotient is written as $L(c,h)$.
		\item $V(c,0)$ can be endowed with the structure of an OZ VOA generated by its conformal vector $\ccv$ with $Y(\ccv,z) = \sum_{n\in \Z}\L n z^{-n-2}$,
		\item The simple quotient $L(c,0)$ is a simple VOA quotient of $V(c,0)$. 
		\end{enumerate}
\end{theorem}
The constant $c$ is called the \emph{central charge} and the constant $h$ is called the \emph{weight} of the module.

As in \cite{M96}, we will use the structure theory of modules of the Virasoro algebra $L(c,0)$ with central charge $c= \frac{1}{2}$ to determine the fusion law of certain idempotents in the algebras under discussion. 
\begin{definition}\label{def:Ising} Let $V$ be a VOA of strong CFT type, and $e\in V_2$. Write $\mathrm{Vir}(\iv)$ for the vertex subalgebra generated by $\iv\in V_2$.  We call $\iv$  an \emph{Ising vector} if $\mathrm{Vir}(\iv) \cong L(\frac{1}{2},0)$, where $\iv$ maps to the conformal vector of $L(\frac{1}{2},0)$.
\end{definition}
 Ising vectors have received much attention in the literature, see \cite{M96,DR16,JLY1,JLY2,Mat05}. We can now state the main proposition of the preliminaries. 
 \begin{proposition}\label{prop:fuse}
 Let $V$ be a VOA over a field $\mathbb{F}$ with $\operatorname{char} \mathbb{F}\neq 2,7$, and $\iv$ an Ising vector vector in $V_2$. Let $W\leq V$ be the  vertex subalgebra generated by $\iv$. Then we have the following:
 \begin{enumerate}
 \item The $W$-module $V$ is completely reducible into irreducible $W$-modules, which are isomorphic to $L(\frac12,h)$ for $h$ either $ 0$, $\frac12$, or $\frac1{16}$. 
 \item Writing $U_h$ for the $L(\frac12,h)$-isotypic component of $V$, we have
 \begin{equation}\label{eq:prop:fuse}
 	Y(U_g,z)U_h \subseteq \bigoplus_{i\in g\star h} U_i[[z,z^{-1}]],
 \end{equation} 
 where $\star$ is the following fusion law: \[\begin{array}{c|ccc}
 \star & 0 & \frac12 & \frac 1{16}	\\\hline
 0 & 0 & \frac 12 &\frac1{16}\\
 \frac12 & \frac 12 & 0 & \frac1{16}\\
 \frac1{16} & \frac1{16} &\frac1{16} & 0,\frac12
 \end{array}\]

 \end{enumerate}
\end{proposition}
 The proof generalizes Miyamoto's arguments in \cite{M96} to arbitrary characteristic, building on the work of Dong and Ren \cite{DR2,DR16}. Since introducing all relevant notions would fall out of the scope of the article, we will be brief in the following proof.
 \begin{proof}
	As $e$ is an Ising vector, the VOA $W$ is isomorphic to the VOA $L(\frac12,0)$. By \cite[Theorem~4.9]{DR16}, $V$ is completely reducible as a $W$-module. By \cite[Theorem~4.6]{DR16}, irreducibles are isomorphic to $L(\frac12,h)$ for $h$ either $ 0$, $\frac12$, or $\frac1{16}$. This proves the first statement.
	 
	 Fix $h_1,h_2,h_3 \in \{0,\frac12,\frac1{16}\}$ arbitrary. Let $M^i\leq V$ be irreducible $W$-submodules isomorphic to $L(\frac12,h_i)$, chosen arbitrary. Then in order to prove the inclusion \eqref{eq:prop:fuse} it suffices to show that
	 \[\pi_{M^3}(Y(M^1,z)M^2) = 0 \text{ if } h_3 \not\in h_1\star h_2,\] where $\pi_{M^3}$ is the projection onto the module $M^3$. One can assume the $M^i$ are \emph{homogeneous}, i.e., that the operator $\L0 = \ccv_1$ acts internally on $M^i$, because the actions of $\L 0$ and $e_1$ commute by the following computation:
	 \begin{align*}
	 	[\ccv_1 ,e_1] &= (\ccv_0 e)_2 + (\ccv_1 e)_1 &{\quad \text{(\cref{lem:formulas}(iii))}}\\
	 							&= -2e_1 + 2e_1 =0 &{\quad\text{(\cref{lem:formulas}(iv))}}.
	 \end{align*}
	 In particular, $M^i$ is generated by an element $v_i$ such that $e_nv^i =0$ for $n>1$, $e_1v^i = h_iv^i$ and $L_0v_i = m_iv^i$ for some $m_i\in \Z$ (resp.\@ $\in \mathbb{Q}$) when $\operatorname{char}\mathbb{F} > 0  $ (resp.\@ $\operatorname{char}\mathbb{F}=0$).
	 
	 To prove the claim now, one can argue as in \cite[Proposition~4.4]{M96}. Briefly, as in \cite[Lemma~4.1]{M96}, the operator $\ccv_1-e_1$ commutes with the action of $\Vir(e)$. Since it also acts as a scalar $\lambda_i = h_i - m_i$ on each $v^i$, it acts as scalars  $\lambda_i \in  \Z$ (resp. $\in \mathbb{Q}$) when $\operatorname{char} \mathbb{F} >0$ (resp. $=0$) on $M^i$.  
	 
	 Then the map 
	 \[M^1 \to \Hom(M^2,M^3)\{z\} \colon v \mapsto \pi_{M^3}\circ Y(v,z)z^{\lambda_3-\lambda_2-\lambda_1}\]
	  is an $\binom{M^1 \ M^2}{M^3}$-intertwiner in the sense of \cite[p.\@  503]{DR16}. Hence this map should be identically zero whenever $h_3 \not \in h_1 \star h_2 $ by \cite[Theorem~6.5]{DR2}. This proves the second claim.
 \end{proof}
 \section{Generalized Griess algebras}
In this section we generalize the notion of a Griess algebra to the non-OZ setting. Throughout this section we assume that $V$ is a VOA of strong CFT type. Recall that the \emph{Gries algebra} of an OZ VOA $V$ is the degree two space $V_2$ together with the product $a,b \to a_1b$. Borcherds \cite[Section 9]{B86} suggests two methods to construct algebras from $V$. 

The first method defines products $*_k$ by $a*_kb \coloneqq a_{k-1} b$ on the graded component $V_k$. The grading on $V$ ensures that $*_k$ is internal. In general, these products are not commutative.
The second method defines products $\times_k$ on $V_{k+2}$ by 
\begin{equation}\label{eq:timesk}
	a\times_k b \coloneqq \sum_{i\geq 0} \frac{(-1)^i}{i+1} \D i a_{k+i+1}b.
\end{equation}
This product is commutative (resp.\@ anti-commutative) when $k$ is even (resp.\@  odd).
We will only be interested in $\times_0$, restricted to the weight $2$ subspace. In that case, it suffices to require $\operatorname{char} \mathbb{F} \neq 2$ for \eqref{eq:timesk} to make sense, since the terms in $i>1$ disappear. 
 For further reference, observe that 
 \[a\times_0 b = \frac{1}{2}(a_1b + b_1a),\]
 by applying skew-symmetry (\cref{lem:formulas}(i)).
When $V$ is OZ, the maps  $*_2$ and $\times_0$ coincide, and  $*_2$ is precisely the multiplication of the corresponding Griess algebra. In the non-OZ case, $\times_0$ seems to be the more natural product to consider.
 
It turns out that $\times_0$ restricts nicely to a natural subspace of $V_2$. Observe the following decomposition.
\begin{lemma}\label{eq:decomposition-V2}
\begin{equation*}
	V_2 = \L{-1} V_1 \oplus \operatorname{ker}{\L1|_{V_2}}.
\end{equation*}
\end{lemma}
\begin{proof}
	We have $V_2 = \L{-1} V_1 + \operatorname{ker}{\L1|_{V_2}}$ since   $v = \frac{1}{4}\L{-1}\L1v + (v-\frac{1}{4}\L{-1}\L1v)$ for any $v\in V_2$. Note that  $\L1 \L{-1}v = [\L1\L{-1}] = 2v$ for all $v\in V_1$ by (CFT1), (VOA1) and the Virasoro commutation relations \eqref{eq:vir},  while $\L1 v =0$  for $v\in  \operatorname{ker}{\L1|_{V_2}}$. This implies the sum in the statement of the lemma is direct.
\end{proof}

Moreover, $\times_0$ behaves well with respect to this decomposition, and the multiplication on $\L{-1} V_1$ is rather uninteresting.
\begin{lemma}\label{lemma:invariant-subspace}
	Let $V$ be a VOA of strong CFT type. The product $\times_0$ is zero on $\L{-1} V_1$ and internal on $\operatorname{ker}{{\L1 }|_{V_2}}$.
\end{lemma}
\begin{proof} The proof is a computation using the formulas from \Cref{lem:formulas}.
For $\ViA,\ViB\in V_1$, the expression \[(\L{-1} \ViA)_1\L{-1} \ViB \stackrel{\text{\ref{lem:formulas}(iv)}}{=} -\ViA_0\L{-1} \ViB=[\L{-1},\ViA_0]\ViB-\L{-1}\ViA_0\ViB\stackrel{\text{\ref{lem:formulas}(iv)}}{=} -\L{-1}\ViA_0\ViB \] is anti-symmetric in $\ViA,\ViB$ by \cref{L1V10}, hence $(\L{-1} \ViA)_1\L{-1} \ViB + (\L{-1} \ViB)_1\L{-1} \ViA = 0$. This proves the first statement. 

For $\ViiA,b\in \operatorname{ker}{\L{1}|_{V_2}}$ we have
\[\L1\ViiA_1\ViiB = [\L1,\ViiA_1]\ViiB  = \ViiA_2\ViiB\]
by \cref{L1V10}(iii) and $\ViiA_2\ViiB = -\ViiB_2\ViiA$ by \cref{L1V10}(ii). We can use this to compute
\[ \L1 (\ViiA \times_0 \ViiB) =\frac{1}{2} (\ViiA_2\ViiB + \ViiB_2\ViiA )= 0, \]
concluding the proof of the second claim.
\end{proof}
\begin{definition}\label{Generalized Griess algebra}
For a VOA $V$ of strong CFT type, we call the space $G(V)\coloneqq \operatorname{ker}{\L1|_{V_2}}$ together with the product $\times_0$ and the bilinear form defined by $(a,b)\vacuum = a_3b$ a \emph{generalized Griess algebra}.
\end{definition}
\begin{remark}
When $V$ is OZ, we recover the original definition by Miyamoto, since $\L1 (V_2)\leq V_1=0 $. 
\end{remark}
In fact, this bilinear form $(\cdot,\cdot)$ is associative in the following sense.
\begin{lemma}\label{lemma:frob_form}
The bilinear form $(\cdot,\cdot)$ on a generalized Gries algebra $G(V)$ is symmetric and satisfies \[(a,b\times_0c) = (a\times_0 b,c),\] for $a$, $b$, and $c\in G(V)$. 
\end{lemma}
\begin{proof}
	The fact that $(\cdot,\cdot)$ is symmetric follows from \cref{L1V10}(i).
	
	 We have $x_3y_0z\in V_1$ for all $x,y,z\in G(V)$ by (VOA4), hence $\L1(x_3y_0z) = 0$ by (CFT1). This means that 
	\begin{equation}\label{eq:frobsub1}
		[\L1,x_3]y_0z = -x_3\L1y_0z = -x_3[\L1,y_0]z,
	\end{equation}
	where on the right-hand side, we used that $\L1z=0$. Applying \cref{L1V10}(iii) to both sides of \eqref{eq:frobsub1}, we obtain that \begin{equation}\label{eq:frobsub2}
		x_4y_0z = 2x_3y_1z \text{ for all }x,y,z\in G(V).
	\end{equation}

	By essentially the same argument, we have \begin{equation}\label{eq:frobsubforgot}
		0= [\L1,x_2]y_1z = -x_2[\L1,y_1]z = -x_2y_2z \text{ for all }x,y,z\in G(V).
	\end{equation}
	
	Now, for $x,y,z\in G(V)$ the iterate formula (\cref{lem:formulas}(ii)) gives:
	\begin{equation*}
		(x_1y)_3z = x_1y_3z+y_4x_0z-x_0y_4z-y_3x_1z.
	\end{equation*}	
	We have $y_4z \in V_{-1}$ and $y_3z\in V_0$, hence $x_0y_4z=0$ and also $x_1y_3z=0$ by (VA3). So,  by \eqref{eq:frobsub2}, we have
	\begin{equation}\label{eq:frobsub3}
			(x_1y)_3z = y_3x_1z \text{ for all } x,y,z \in G(V).
	\end{equation}
	We will use this to show that $(a_1b)_3c$ is invariant under permutations of $a$, $b$ and $c$, for any $a$, $b$ and $c$ in $G(V)$.	On the one hand, we already know that $(a_1b)_3c = c_3a_1b = (a_1c)_3b$. On the other hand, we have
	\begin{align*}
		(a_1b)_3c &= b_3a_1c &(\text{\cref{eq:frobsub3}})\\
		&= b_3c_1a - b_3\L{-1} c_2a &(\text{\cref{lem:formulas}(i)})\\
		&= b_3c_1a + [\L{-1},b_3] c_2a &(\text{(VA3) and (VOA4)}) \\
		&= b_3c_1a -3b_2c_2a &(\text{\cref{lem:formulas}(iv)}) \\
		&= b_3c_1a &(\text{\cref{eq:frobsubforgot}}).
	\end{align*}
	So the expression $(a_1b)_3c$ is invariant under permutations of $a$, $b$, and $c$. We can use this to prove the claim as follows:
		\begin{equation*}
		2(a\times_0b)_3c = (a_1b)_3c + (b_1a)_3c = (b_1c)_3a +(c_1b)_3a  = 2(b\times_0c)_3a. \qedhere
	\end{equation*}
\end{proof}

We end this section with the observation that gave the idea for this work.
\begin{proposition}\label{prop:CGisgriess}
	Chayet--Garibaldi algebras are generalized Griess algebras.
\end{proposition}
\begin{proof}
	This follows from \cite[Theorem~5.23]{DMO}, by noting that the subspace $L^{\mathrm{sym}}_{(2)}\leq L_{\hat{g}}(1,0)$ is precisely the kernel of $\L1$ restricted to the weight~$2$ component $V_2$ of $V$, where the conformal vector is the usual one (see \cite[Equation (6.2.42)]{LL04}).
\end{proof}

\section{Ising vectors in generalized Griess algebras}

In OZ VOAs, Ising vectors $e$ yield idempotents $\frac{1}{2}e$ of the corresponding Griess algebra. Miyamoto proved in \cite{M96} that these idempotents satisfy the Monster fusion law. We prove an analogue for generalized Griess algebras $G(V)$. First, we prove that any Ising vector $\iv$ of $V$ is contained in $G(V)$ (\cref{lemma:42}). Next, we compute the eigenspaces corresponding to left multiplication by the Ising vector (\cref{lemma:eigenspaces}). Finally, we use the fusion rules of the Virasoro VOA to compute the fusion law of the Ising vector (\cref{thm:fusion_law,prop:fuslawBn}).

Recall that we work over a field $\mathbb{F}$ of characteristic not $2$ or $7$. We will also assume the characteristic is not $3$ or $31$ to avoid some technicalities. We outline how to deal with these cases in \cref{rem:kar31}.  Throughout this section, $(V,Y,\vacuum)$ is a vertex operator algebra of strong CFT type and $\iv$ an Ising vector of $V$. 

Recall that $\mathrm{Vir}(\iv)$ is isomorphic to the irreducible Virasoro VOA of central charge $\frac{1}{2}$. In particular, all its modules are completely reducible, and its irreducible modules are known: they are all generated by a lowest weight vector of weight $0$, $\frac12$ or $\frac1{16}$ (\cite[Theorem~4.9]{DR16}). 

Let $v\in V$ be a lowest weight vector of $\mathrm{Vir}(\iv)$ with $\iv_1v=hv$ and $\iv_nv=0$ for $n>1$. Write $v= v_0 + \dots +v_k$ with $v_i\in V_i$ the homogeneous components of $v$. Since $e_n$ shifts the grading of $V$ by $n-1$ by (VOA4), it is easy to see that each $v_i$ is also a lowest weight vector for $\mathrm{Vir}(\iv)$.

Hence, the completely reducible $\mathrm{Vir}(\iv)$-module $V$ can be decomposed into \emph{homogeneous} irreducible modules. That is, it can be decomposed into modules isomorphic to $L(\frac{1}{2},h)$ whose lowest weight vector is homogeneous with respect to the grading $V=\bigoplus_{n\in \mathbb{N}}V_n$.

In particular, we have the following: $V_0=\mathbb{F}\vacuum$ is spanned by a lowest weight vector for $\mathrm{Vir}(\iv)$. Since $e_0\vacuum =0$ and $V$ is a direct sum of homogeneous $\mathrm{Vir}(\iv)$-modules, $V_1$ is necessarily also spanned by lowest weight vectors.  Hence we obtain the decompositions
\begin{align}\label{decompositio}
V_1 &= \bigoplus_{h=0,\frac{1}{2},\frac{1}{16}} T_{h,1},\\ V_2 &= \mathbb{F}\iv \oplus \iv_0(V_1)\oplus \bigoplus_{h=0,\frac{1}{2},\frac{1}{16}} T_{h,2} ,\label{decomposition}
\end{align}
where 
\[T_{h,n} \coloneqq \{v \in V_n\mid  \iv_{1}v = hv\text{, and } \iv_iv = 0 \text{ for } i\geq 2\}.\]

We can now prove the announced lemma.
\begin{lemma}\label{lemma:42} 
Any Ising vector $\iv\in V$ is contained in $\ker \L1$: 
\[\L1 \iv  = 0.\]	
\end{lemma}
\begin{proof}
	By the commutator formula, we have 
	\begin{align*}
		[\L1,\iv_1] \iv&= (\ccv_0 \iv) _3\iv + 2(\ccv_1\iv)_2\iv + (\ccv_2\iv)_1\iv \\
		&= (\L{-1}\iv)_3\iv + 2(\L0\iv)_2\iv + (\L1\iv)_1\iv  \\ 
		&= -3\iv_2\iv+4\iv_2\iv + (\L1\iv)_1\iv.
	\end{align*}
	In the last equality, we used \cref{lem:formulas}(iv) for the first term and (VOA2) for the second. Now $\iv_2\iv=0$ by \cref{L1V10}(ii), which gives $[\L1,\iv_1] \iv= (\L1\iv)_1\iv = \iv_1\L1\iv -\mathcal{D}(\iv_2\L1\iv ) = \iv_1\L1\iv$, since $\iv_2\L1\iv \in V_0$ and $\mathcal{D}(V_0)=0$.
	We now compute the image of $\L1\iv$ under the map $\iv_1$:
	\begin{align*}
		\iv_1\L1\iv &= [\iv_1,\L1]\iv+\L1\iv_1\iv \\
		&= - \iv_1\L1\iv + 2\L1\iv = \L1\iv. 
	\end{align*}
	This shows that $\L1\iv\in V_1$ is an eigenvector with eigenvalue $1$ for $\iv_1$. However, the space $V_{1}= \bigoplus_{h\in \{0,\frac12,\frac1{16}\}} T_{h,1}$ is spanned by eigenvectors with eigenvalues $0$, $\frac12$ and $\frac1{16}$. Hence $\L1\iv$ has to be zero.
\end{proof}

\subsection{Eigenspaces}
We use the decompositions \eqref{decompositio} and \eqref{decomposition} to compute the eigenspaces of the operator
\[R_\iv: V_2 \to V_2 \colon v  \mapsto \iv\times_0 v.\] This operator can also be written as \[R_\iv = \iv_1 - \frac12\mathcal{D}\iv_2.\] 
 \begin{lemma}\label{lemma:eigenspaces}
	The operator $R_\iv$ is semisimple and has the following eigenspaces in $V_2$:
	\[\begin{array}{c|c}
			\text{eigenvalue} & \text{eigenspace} \\\hline
			0 & T_{0,2}\\
			\frac1{16} & T_{\frac1{16},2}\\
			\frac12 & T_{\frac12,2}\\
			1 & S \coloneqq  \{ v_0\iv \mid v\in V_1\}\\
			2 & \langle \iv \rangle.\\
	\end{array}\]
\end{lemma}
\begin{proof}
	First, we prove that the spaces listed are eigenspaces and then that they span $V_2$.
	\paragraph{\underline{Eigenvalues $0$, $\frac12$, and $\frac1{16}$}}
	Let $v$ be a vector in $T_{h,2}$ with $h\in \{0,\frac{1}{2},\frac{1}{16}\}$. Then we have $\iv_2v=0$. We compute that \[R_\iv v = \iv_1 v - \frac12 \mathcal{D}(0) = hv.\]
	\paragraph{\underline{Eigenvalue $2$}} We have $\iv_2\iv=0$ and $\iv_1 \iv = 2\iv$. This implies that $R_\iv v = 2\iv$.
	\paragraph{\underline{Eigenvalue $1$}} Let $v$ be a vector in $V_1$. By skew-symmetry (\cref{lem:formulas}(i)), we have 
	\[R_\iv (u_0\iv) = \frac{1}{2}(\iv_1u_0\iv + (u_0\iv)_1\iv).\]
	We compute further,  using the iterate formula (\cref{lem:formulas}(ii)), that
	\begin{align*}
		R_\iv (u_0\iv) &= \frac{1}{2}(\iv_1u_0\iv + (u_0\iv)_1\iv) \\
		&= \frac{1}{2}(\iv_1u_0\iv + u_0\iv_1\iv - \iv_1u_0\iv) = u_0\iv,
	\end{align*}
	where in the last equality, we used the fact that $\iv_1\iv = 2\iv$.\medskip
  
 Next we prove that $V_2= \langle \iv \rangle \oplus S \oplus T_{0,2}\oplus T_{\frac12,2}\oplus T_{\frac{1}{16},2}$. By decomposition \eqref{decomposition}, it suffices to show that the kernel of the map $\iv_2\colon V_2 \to V_1$ contains $\mathbb{F}\iv\oplus_h T_{h,2}$, while both $S$ and $\iv_0(V_1)$ map isomorphically onto their image.  The first claim follows from the fact that $\iv_2\iv = \iv_2v=0$ for any lowest weight vector $v$. For the second claim, first observe that $S$ decomposes as a sum \begin{equation}\label{eq:S}S_0 + S_{\frac12} + S_{\frac1{16}}\end{equation} with $S_h=\{v_0\iv\mid v\in T_{h,1}\}$. 
  
  We prove that the image of $S_h$ under the map $\iv_2$ equals the image of $\iv_0T_{h,1}$. If $h=0$, then we have $\iv_1T_{0,1}=0$ and $\iv_0T_{0,1}=0$ by \cref{thm:virvoa}. Hence, $\iv_0T_{0,1} =0$, and by skew-symmetry (\cref{lem:formulas}(i)) also $S_0=0$.
  
  If $h\not=0$, then we prove $\iv_2$ defines isomorphisms from $\iv_0T_{h,1}$, and $S_h$ onto $T_{h,1}$. We have for $v\in T_{h,1}$ that respectively
  \begin{align*}
  	\iv_2\iv_0v &= [\iv_2,\iv_0]v \\ &= (\iv_0\iv)_2v +2(\iv_1\iv)_1v + (\iv_2\iv)_0v \\&= 2\iv_1v = 2hv,\intertext{and}
  	\iv_2v_0\iv &= \iv_2\iv_0v - \iv_2\mathcal{D}(\iv_1v) \\&= 2\iv_1v - [\iv_2,\mathcal{D}]\iv_1v \\&= 2\iv_1v -2\iv_1\iv_1v=2(h-h^2)v.
  \end{align*}
This proves our claim that $\iv_2$ maps both $S$ and $\iv_0(V_1)$ isomorphically onto $\oplus_{h\in\{\frac{1}{2},\frac{1}{16}\}} V_{h,1}$. From this it follows that $V_2 = S\oplus \mathbb{F}\iv \oplus \bigoplus_{h} T_{h,2}$, proving the claim.  
\end{proof}
\begin{remark}\label{rem:e2inj}
	In the course of the proof, we showed that $e_2$ maps $S$ injectively into $V_1$. We will use this fact again in \cref{lem:1/2}.
\end{remark}
These eigenspaces restrict nicely to the generalized Griess algebra $G(V)$.

\begin{lemma}
	The operator $R_\iv$ is diagonalizable on $G(V)=\ker(\L1|_{V_2})$ with spectrum a subset of $\{0,\frac1{16},\frac12,1,2\}$.
\end{lemma}
\begin{proof}
By \cref{lemma:invariant-subspace,lemma:42}, the space $G(V)$ is stable under multiplication by $\iv$:
\[ R_\iv \ker(\L1) \leq \ker(\L1).\]
 The map $R_\iv$ is semisimple on $V_2$ by the previous lemma. Hence $R_\iv|_{G(V)}$ is semisimple as the restriction of a semisimple map to a stable subspace. The spectrum is then a subset of $\{0,\frac1{16},\frac12,1,2\}$.
\end{proof}
Not only that, but it turns out that the $1$-eigenspace is fully contained in $G(V)$, and can be related to derivations of $G(V)$.
\begin{lemma}\label{lem:derivations}
	For any $u\in V_1$, the operator $u_0\colon V_2 \to V_2$ restricts to a derivation on $G(V)$.  In particular, $S\subseteq G(V)$, and the $1$-eigenspace of $R_\iv$ is equal to $S= \langle D(\iv) \mid D\in \mathrm{Der}(G(V))\rangle$.
\end{lemma}
\begin{proof}
	Let $u\in V_1$ be arbitrary. For any $b\in \ker \L1$, we have
	\begin{align*}
		\L1 u_0 b &= [\L1,u_0]b\\
		&= (\L{-1}u)_2b+ 2(\L0u)_1b +(\L1u)_0b \\
		&= -2u_1b +2u_1b +0 =0. 
	\end{align*}
	This shows that $u_0$ restricts to an operator on $G(V)$. We can also verify that it satisfies the Leibniz rule: 
	\begin{multline*}
		u_0(a\times_0b) = \frac{1}{2}u_0(a_1b + b_1a) = \frac{1}{2}([u_0,a_1]b + a_1 (u_0b) + [u_0,b_1]a + b_1(u_0a) \\
		=\frac{1}{2}((u_0a)_1b + a_1 (u_0b) + (u_0b)_1a + b_1(u_0a)) = u_0a \times_0 b + a \times_0 u_0b.
	\end{multline*} 
	Hence, we have $u_0\in \mathrm{Der}(G(V))$. The fact that $S\subseteq G(V)$ follows, since $e\in G(V)$. For the last claim, it can be seen that for any derivation $D\in \mathrm{Der}(G(V))$, the element $D(\iv)$ is a $1$-eigenvector of $R_e$, by the Leibniz rule:
	\[2\iv\times_0 u_0\iv = u_0(\iv \times_0 \iv) = 2u_0\iv.\qedhere\]
\end{proof}
\begin{remark}
	In a sense, the above lemma also tells us that the $1$-eigenspace is minimal, from the perspective of the derivation algebra $\mathrm{Der}(G(V))$. That is, any potential candidate for the image of $\iv$ under a derivation of $G(V)$ has to be a $1$-eigenvector, and the above shows that this exhausts the entire $1$-eigenspace.
\end{remark}
%
%
%
%

\subsection{Fusion law} We are now ready to compute the fusion law for the eigenspaces of $R_\iv$. To do so, we first introduce some notation, then we use \cref{prop:fuse} to significantly reduce the complexity of the fusion law. We end with a few lemmas that narrow down the fusion law even further.

We keep the notation from the beginning of this section: $V$ is a VOA; $\iv$ is an Ising vector; and $T_{0,2}$, $T_{\frac{1}{16},2}$, $T_{\frac{1}{2},2}$, $S$ and $\mathbb{F} \iv $ are the eigenspaces of $R_\iv$. We denote these eigenspaces as $A_0$, $A_{\frac1{16}}$, $A_{\frac1{2}}$, $A_1$, and $ A_2$ respectively to keep the notation from becoming too cumbersome. We write \[*:\operatorname{Spec}(R_\iv)\times \operatorname{Spec}(R_\iv) \to 2^{\operatorname{Spec}(R_\iv)}\] for the (minimal) fusion law of the eigenspaces: so $\lambda\in\mu *\nu$ if and only if the projection of $A_\mu\times_0A_\nu$ onto $A_\lambda$ is nonzero. 
We immediately have the following properties.
\begin{lemma}
	The operation $*$ is symmetric and we have that 
\[2 *\lambda = \{\lambda\} \quad \text{for} \quad \lambda\in\operatorname{Spec}(R_\iv)\backslash \{0\},\] and $2*0=\emptyset$.
\end{lemma}
\begin{proof}
This follows directly because $\times_0$ is commutative and multiplication by $\iv$ preserves the eigenspaces.	
\end{proof}

Next, we use the representation theory of the irreducible Virasoro VOA $W$ of central charge $\frac{1}{2}$ to significantly reduce the number of option for $*$.
\begin{lemma}\label{lem:stdfusion} The fusion law $*$ satisfies the following properties:
\begin{align*}
	                  0*0 &\subseteq \left\{0,1,2\right\} &
	            \frac12*0 &\subseteq \left\{\frac12,1\right\}, &
	         \frac1{16}*0 &\subseteq \left\{\frac1{16},1\right\}, \\
	      \frac12*\frac12 &\subseteq \left\{0,1,2\right\}, &
 	   \frac12*\frac1{16} &\subseteq \left\{\frac1{16},1\right\}, &
	\frac1{16}*\frac1{16} &\subseteq \left\{0,\frac12,1,2\right\}
\end{align*}
\end{lemma}
\begin{proof} We have to compute the fusion law of $T_{\mu,2}$ and $T_{\nu,2}$ for $\mu,\nu\in\{0,\frac12,\frac1{16}\}$. Denote the fusion law from \cref{prop:fuse} by $\star$. By the same proposition we have \[T_{\mu,2} \times_0 T_{\nu,2} \subset \bigoplus_{\lambda \in \mu\star\nu}\iv_{-1} T_{\lambda,0}\oplus \iv_0T_{\lambda,1}\oplus T_{\lambda,2}.\] In other words $\mu * \nu$, is a subset of $ (\mu\star\nu) \cup \{1,2\}$, and $2$ is only present if $0\in \mu\star \nu$, since $\iv_{-1} T_{\lambda,0}$ is non-zero only when $\lambda = 0$.
\end{proof}

Next, we prove some specific lemmas to further restrict the fusion law. 
\begin{lemma}
Let $\lambda,\mu\in \{0,\frac1{16},\frac12,1,2\}$. If $2\in \lambda * \mu$, then $\lambda\neq 0$ and  $\lambda = \mu$.
\end{lemma}
\begin{proof}
	Take two vectors $u \in A_\lambda$ and $v\in A_\mu$.
	Use $\pi\colon A \to \mathbb{F} \iv$ to denote the projection onto the $2$-eigenspace $A_2$. 
	Suppose that $\pi( u\times_0 v) = \rho\iv$ for some nonzero $\rho\in \mathbb{F}$. Recall the associative bilinear form from \cref{lemma:frob_form}. We have $(\iv, u\times_0 v) = \frac{1}{2}\rho$. However, since $(\cdot,\cdot)$ is associative, this is equal to $ \frac{1}{2}\rho = \lambda(u,v) = \mu(u,v)$, whence the lemma.
\end{proof}
%
\begin{lemma}\label{lem:1/2one}
	Let $\lambda \in \{0,\frac1{16},\frac12,1,2\}$. Then $\lambda \notin 1*\lambda$.
\end{lemma}
\begin{proof}
	Take $u\in A_1$, so it is an arbitrary $1$-eigenvector of $R_\iv$. By \cref{lem:derivations}, the element $u$ is equal to $D(\iv)$ for some $D\in \mathrm{Der}(A)$.
	
	If $v\in A_\lambda$ is an arbitrary $\lambda$-eigenvector, then we can use the Leibniz rule to compute $u \times_0 v$:
	\begin{multline*}
		u \times_0 v =  D(\iv)\times_0 v  = D(\iv\times_0 v) - \iv\times_0 D(v) = \lambda D(v) - R_\iv D(v).
	\end{multline*}
	Write $D(v)_\lambda$ for the projection of $D(v)$ onto $A_\lambda$. Then the projection of $u \times_0 v$ onto $A_\lambda$ is $\lambda D(v)_\lambda - R_\iv D(v)_\lambda =0$, hence the product $u \times_0 v$ does not contain an $A_\lambda$-summand.
\end{proof}
%
\begin{lemma}\label{lem:1/2}
	Let $\lambda\in \{0,\frac1{16},\frac12\}$. Then $1 \notin \lambda*\lambda$.
\end{lemma}
\begin{proof}
Take $u,v\in A_\lambda$. We prove first that $\iv_2(u\times_0 v) = 0$ using the following formulas. We have that
	$\iv_2u = 0 = \iv_2 v$, and $\iv_1u=\lambda u$,
because $u$ and $v$ are lowest weight vectors for $\iv$. 

By \cref{L1V10}(i) and (ii) we have
\[
	u_2v=-v_2u\text{, and } (\iv_0v)_3u=u_3\iv_0v.
\]
The commutator formula (\cref{lem:formulas}(iii)) yields
	\[ [\iv_2,u_1]=(\iv_0u)_3+2(\iv_1u)_2+(\iv_2u)_1 = (\iv_0u)_3+2\lambda u_2.\]
We can finally compute $\iv_2(u\times_0 v)$. In the following, we underline groupings of terms that equal zero:
\begin{align*}
	2\iv_2(u\times_0 v) &= \iv_2u_1v+\iv_2v_1u \\
	&= \underline{u_1\iv_2v}+\underline{v_1\iv_2u}+[\iv_2,u_1]v+[\iv_2,v_1]u \\
	&=(\iv_0u)_3v+(\iv_0v)_3u+\underline{2\lambda u_2v+2\lambda v_2u} \\
	&=\underline{\iv_0u_3v}-u_3\iv_0v+(\iv_0v)_3u\\
	&=\underline{(\iv_0v)_3u-u_3\iv_0v}=0.
\end{align*}

The map $\iv_2$ is an injection from $A_1$ to $V_1$ by \cref{rem:e2inj}, so the projection of $u\times_0 v$ onto $A_1$ is zero.
\end{proof}
When we combine all the previous lemmas we obtain the following fusion law.
\begin{theorem}\label{thm:fusion_law}
	Let $V$ be a VOA of strong CFT type and $\iv\in V$ an Ising vector. Then $\iv$ satisfies the following fusion law in $(G(V),\times_0)$: \[	\renewcommand{\arraystretch}{1.5}
	\begin{array}{c||c|c|c|c|c}
	
 *          & 0                    & \frac12        & \frac1{16}   & 1                      & 2          \\\hline\hline
 0          & 0                    & \frac12,1      & \frac1{16},1 & \frac1{16},\frac12,1   &            \\\hline
 \frac12    & \frac12,1            & 0,2            & \frac1{16},1 & 0,\frac1{16},1         & \frac12    \\\hline
 \frac1{16} & \frac1{16},1         & \frac1{16},1   & 0,\frac12,2  & 0,\frac12,1            & \frac1{16} \\\hline
 1          & \frac1{16},\frac12,1 & 0,\frac1{16},1 & 0,\frac12,1  & 0,\frac12,\frac1{16},2 & 1          \\\hline
 2          &                      & \frac12        & \frac1{16}   & 1                      & 2
 
	\end{array}. 
	\]	
 Each cell in this table, indexed by row $\lambda$ and column $\mu$, contains the elements that can occur in $\lambda * \mu$.

\end{theorem}

\subsection{Ising vectors of $\sigma$-type} From the perspective of automorphisms, the above fusion law is rather disappointing, as it is not graded. However, by restricting to Ising vectors of $\sigma$-type, we recover the possibility of studying automorphisms. An Ising vector $\iv$ is called of \emph{$\sigma$-type} if the module $L(\frac12,\frac{1}{16})$ does not occur as a $\mathrm{Vir}(\iv)$-submodule of $V$.

\begin{proposition}\label{prop:fuslawBn}
	If $e$ is an Ising vector of $\sigma$-type, then $e$ satisfies the following fusion law in $(G(V),\times_0)$:
\[	\renewcommand{\arraystretch}{1.5}
	\begin{array}{c||c|c|c|c}
	
 *          & 0         & \frac12   & 1         & 2          \\\hline\hline
 0          & 0         & \frac12,1 & \frac12,1 &            \\\hline
 \frac12    & \frac12,1 & 0,2       & 0         & \frac12    \\\hline
 1          & \frac12,1 & 0         & 0,2       & 1          \\\hline
 2          &           & \frac12   & 1         & 2
 
	\end{array}. 
	\]	
	
\end{proposition}
\begin{proof}
Note that we can first remove all the $\frac1{16}$ occurences from the previous table. Because $L(\frac{1}{2},\frac{1}{16})$ does not occur as a submodule, we can also see that the $1$-eigenspace of $R_\iv$ is contained in $U_{\frac{1}{2},1}$.
	Hence, analogously to \cref{lem:stdfusion}, we can use \cref{prop:fuse} to conclude that 
	$
 	1*\frac12 \subset \left\{ 0,2\right\}$  and  $1*1 \subset \left\{ 0,2\right\}$.
\end{proof}

\begin{remark}\label{rem:kar31}
	The reason for excluding characteristics $3$ and $31$ should be clear: in those cases, either $2=\frac12$ or $2=\frac1{16}$. Then the Ising vector $e$ is no longer \emph{primitive}, i.e., its $2$-eigenspace is no longer spanned by $e$, and the fusion laws in \cref{thm:fusion_law,prop:fuslawBn} have repeated eigenvalues. The fusion laws however still make sense in the context of \emph{decomposition algebras}, cf.\@ \cite{DMPSVC}. We leave it to the reader to check the necessary details.
\end{remark}
\begin{remark}
	The fusion laws announced in the introduction (\cref{table:fl}) look a bit different from the ones in \cref{thm:fusion_law,prop:fuslawBn}. This is because the introduction is written with respect to idempotents, while in this section, we stick to Ising vectors. These are related by a scalar $\frac{1}{2}$, so all eigenvalues should be adjusted accordingly. 
\end{remark}

\subsection{Minimality.} The remainder of this section will be dedicated to showing that both fusion laws are minimal, i.e., that no further entries can be omitted without further assumptions on $V$. 

Let $L$ be an even lattice, $\mathbb{F}$ a field, and let $V_{L,\mathbb{F}}$ be the lattice vertex algebra over $\mathbb{F}$ as in \cite{M14}, i.e., $V_{L,\mathbb{F}} = V_{L,\Z}\otimes_\Z \mathbb{F}$, where $V_{L,\Z}$ is the integral form of the lattice vertex algebra constructed by Dong and Griess in \cite{DG12}. The vertex algebra $V_{L,\mathbb{F}}$ contains elements $e^\alpha$ which are not uniquely determined, but depend on the choice of a $2$-cocycle $\epsilon\colon L\times L \to \{\pm 1\}\subset \mathbb{F}$ as in \cite[Section~6.4 and Proposition~6.5.5]{LL04}. We will not delve deeper into this topic, but it suffices for our purposes that given $\alpha\in L$ such a cocycle $\epsilon$ exists with $\epsilon(\alpha,-\alpha)=1$, by using the methods in the proof  \cite[Proposition~5.2.3]{FLM}, cf.\@ \cite[Remark~6.5.6]{LL04}.
 
With this convention, the following was observed by Chayet in the case of Chayet--Garibaldi algebras, but they appeared first in \cite{DLMN} in the context of VOAs.
\begin{lemma}\label{lem:isinglattice}
	Let $V_{L,\mathbb{F}}$ be a lattice VOA with $L$ a positive-definite even lattice, and let $\alpha\in L$ be a lattice vector of length $4$. Then the element
	\[ \iv(\lambda,\alpha) = \frac{1}{16}\alpha(-1)^2 \pm \frac{1}{4}(\lambda e^{\alpha} + \lambda^{-1}e^{-\alpha}),  \]
	with $\lambda \in \mathbb{F}^\times$, is an Ising vector of $V_{L,\mathbb{F}}$.
\end{lemma}
\begin{proof}
	For $\lambda=1$ the fact that $\iv^\pm(1,\alpha)$ are Ising vectors follows from \cite{DLMN}.
	The general case follows from similar computations, or from applying an automorphism of the vertex algebra (after extending the base field). 
\end{proof}
We can investigate when the Ising vectors $e^{\pm}(\lambda,\alpha)$ are of $\sigma$-type. This proposition is probably well-known to experts (see e.g. \cite{DLMN,JLY1}), but we have not found an explicit proof in the literature.
\begin{proposition}
	The Ising vectors $\iv(\lambda,\alpha)$ above are of $\sigma$-type whenever $(\alpha,L) \subseteq 2\Z$.
\end{proposition}
\begin{proof}
	We set without loss of generality $\lambda=1$. By \cref{prop:fuse}, it suffices to prove that a generating set of $V_L$ is contained in the direct sum of the $L(\frac{1}{2},\frac{1}{2})$- and the $L(\frac{1}{2},0)$-isotypic components, denoted $U_\frac{1}{2},U_0$ respectively. By \cite[Theorem 1]{M14}, the set $\{e^{\pm \beta} \mid  \beta \in \Delta \}$ where $\Delta$  is a spanning set of the lattice $L$, is a generating set of $V_{L,\mathbb{F}}$. Without loss of generality, we can choose this spanning set such that $\alpha\in \Delta$, and $\langle \alpha , \beta \rangle \in \{0,2\}$ for $\beta\in \Delta\setminus\{\alpha\}$.
	
	First note that for $\beta \in \Delta$ and $n>0$, we have $(\alpha(-1)^2)_ne^{\beta} =\delta_{n,1}\langle \alpha,\beta \rangle^2e^\beta $. Take $\beta \in \Delta$, we consider three cases.
	\par{Case $\langle \alpha,\beta \rangle = 0$.} We can easily compute from the formula of $Y(e^\alpha,z)$ that also $(e^\alpha)_ne^\beta=(e^{-\alpha})_ne^\beta=0$ for $n \geq 0$.
	 \par{Case $\langle \alpha,\beta \rangle=2$.} Here we have to check this for $e^{\pm\beta}$, but without loss of generality it suffices to compute it for $e^{+\beta}$,  
	 then $(e^\alpha)_ne^\beta=(e^{-\alpha})_{n-1}e^\beta =0$ for $n>0$, while $(e^\alpha)_1e^{\beta} = \epsilon(\alpha,\beta)e^{\alpha+\beta}$. Hence, $e^\beta+\epsilon(\alpha,\beta)e^{\alpha+\beta}$ respectively $e^\beta-\epsilon(\alpha,\beta)e^{\alpha+\beta}$ is a lowest weight vector of weight $\frac{1}{2}$ respectively $0$.

	\par{Case $\beta = \alpha$.} Similarly, $\alpha(-1)^2 - 4(e^\alpha + e^{-\alpha})$ is a lowest weight vector of weight $\frac{1}{2}$, while $\alpha(-1)^2 +4( e^{\alpha} +e^{-\alpha})\in \Vir(e)$, Hence $e^\alpha + e^{-\alpha}$ is contained in $U_\frac{1}{2} \oplus U_0$. On the other hand, $\alpha(-1)$ is a lowest weight vector of weight $\frac{1}{2}$, and $(\iv(1,\alpha))_0\alpha(-1) = \frac{1}{2}\alpha(-2) -(e^\alpha - e^{-\alpha})$, while $\frac{1}{2}\alpha(-2) +(e^\alpha - e^{-\alpha})$ is also a lowest weight vector of weight $\frac{1}{2}$. 
	%
\end{proof}
\begin{remark}
	The same strategy shows that $e^\beta$ with $\langle\alpha , \beta \rangle= \pm 1$ is a lowest weight vector of weight $\frac{1}{16}$. Hence, a lattice VOA contains only Ising vectors of $\sigma$-type if and only if $(L_4,L) \subseteq 2\mathbb{Z}$, where $L_4$ is the set of length $4$ vectors contained in $L$.
\end{remark}

In characteristic zero, the vertex operator algebra $L_{\hat{\mathfrak{g}}}(1,0)$ is isomorphic to a lattice VOA whenever $\mathfrak{g}$ is split simple of ADE type \cite[Proposition~13.10]{DL}. Specifically, $L_{\hat{\mathfrak{g}}}(1,0)\cong V_R$ where $R$ is the corresponding root lattice. Using the previous lemma and proposition, one can easily show that Chayet--Garibaldi algebras contain Ising vectors, both of $\sigma$-type and not of $\sigma$-type. 

Indeed, let $R$ be a root lattice of type $D_5$, spanned by $\alpha_1,\dots,\alpha_5$, all of length 2. The Dynkin diagram of $D_5$ (\cref{fig:D5}) encodes the inner product of the lattice: two nodes are adjacent if and only if the inner product of the corresponding lattice vectors $\alpha$, $\beta$ satisfy $\langle \alpha, \beta \rangle=1$. Lattice vectors are orthogonal whenever the corresponding nodes are not adjacent.
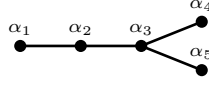
\begin{figure}
	\centering
	\begin{tikzpicture}[line width=1pt, scale=.8, baseline=-.6ex]
		\draw (0,0) -- (1,0);
		\draw (1,0) -- (2,0);
		\draw (2,0) -- (3,.4);
		\draw (2,0) -- (3,-.4);
		\diagnode{(0,0)}{$\alpha_1$} 
		\diagnode{(1,0)}{$\alpha_2$} 
		\diagnode{(2,0)}{$\alpha_3$} 
		\diagnode{(3,.4)}{$\alpha_{4}$} 
		\diagnode{(3,-.4)}{$\alpha_5$} 
	\end{tikzpicture}
\caption{The Dynkin diagram of type $D_5$}\label{fig:D5}
\end{figure}
From this, it is easily seen that $e(1, \alpha_4+\alpha_5)$ is an Ising vector of $\sigma$-type, while $e(1,\alpha_2+ \alpha_5)$ is not of $\sigma$-type. 
 The Magma code accompanying \cite{CG21}, available at \url{github.com/skipgaribaldi/Chayet--Garibaldi}, can then be used to verify that the fusion laws from \cref{thm:fusion_law,prop:fuslawBn} are in fact minimal.
 
\begin{remark}
	\begin{enumerate}
		\item The above argument works for type
		$D_n$ in general. Moreover, in type $B_n$, analogues of the Ising vectors from \cref{lem:isinglattice} can be constructed that are also of $\sigma$-type. So the Chayet--Garibaldi algebras for types $B_n,D_n$ contain Ising vectors of $\sigma$-type. In general, \cref{lem:isinglattice} provides Ising vectors which are not of $\sigma$-type.
		\item The Ising vectors $e(1,\alpha)$ are contained in the Chayet--Garibaldi algebra $A(\mathfrak{g}_\mathbb{R}) \subseteq  A(\mathfrak{g}_\mathbb{C}) $ corresponding to the compact real form $\mathfrak{g}_\mathbb{R}$ of a complex simple Lie algebra $\mathfrak{g}_\mathbb{C}$. It would be interesting to know whether this holds more generally for arbitrary twisted forms.
	\end{enumerate}
	
\end{remark}

\end{document}